\documentclass[11pt]{amsart}
\usepackage{verbatim,amssymb,amsmath,graphicx,hyperref}
\usepackage[usenames]{color}

  \scrollmode

\newcommand{\R}{{\mathbb R}}
\newcommand{\N}{{\mathbb N}}
\newcommand{\Z}{{\mathbb Z}}
\newcommand{\EE}{{\mathbb E}}
\newcommand{\PP}{{\mathbb P}}

\newcommand{\1}{1}

\newcommand{\pp}{\varphi}
\newcommand{\alg}{\mathcal{X}}
\newcommand{\ovw}{v}

\theoremstyle{plain}
\newtheorem{theorem}{Theorem}

\newtheorem{lemma}{Lemma}
\newtheorem{cor}{Corollary}

\theoremstyle{definition}
\newtheorem{rem}{Remark}
\newtheorem{ex}{Example}

\begin{document}
\title[Non-polynomial lower error bounds for strong approximation of SDEs]{On non-polynomial lower error bounds for adaptive strong approximation of SDEs}

\author[Yaroslavtseva]
{Larisa Yaroslavtseva}
\address{
Fakult\"at f\"ur Informatik und Mathematik\\
Universit\"at Passau\\
Innstrasse 33 \\
94032 Passau\\
Germany} \email{larisa.yaroslavtseva@uni-passau.de}

\begin{abstract}
Recently, it has been shown in \cite{hhj12} that there exists a system of stochastic differential equations (SDE) on the time interval $[0,T]$ with infinitely often differentiable and bounded coefficients  such that the Euler scheme with equidistant time steps converges to the solution of this SDE at the final time  in the strong sense but with no polynomial rate. Even worse, in \cite{JMGY15} it has been shown that for any sequence  $(a_n)_{n\in\N}\subset (0,\infty)$, which may converge to zero arbitrary slowly,  there exists an SDE on $[0,T]$ with infinitely often differentiable and bounded coefficients  such that no approximation of the solution of this SDE at the final time
based on $n$
evaluations of the driving Brownian
motion at fixed time points can achieve a smaller absolute mean error  than the given number $a_n$.
 In the present article we  generalize the latter result to the case when the approximations may choose the location as well as the
number of the evaluation sites of the driving Brownian motion in an
adaptive way dependent on the values of the Brownian motion observed so far.

\end{abstract}

\maketitle

\section{Introduction}

Let $d,m\in \N$, $T\in(0, \infty)$, consider a $d$-dimensional system of autonomous stochastic differential equations (SDE)
\begin{equation}\label{sde0}
\begin{aligned}
dX(t) & = \mu(X(t)) \, dt + \sigma(X(t)) \, dW(t), \quad t\in [0,T],\\
X(0) & = x_0
\end{aligned}
\end{equation}
with a deterministic initial value $x_0\in\R^d$, a drift coefficient $\mu\colon\R^d\to\R^d$, a diffusion coefficient $\sigma\colon \R^d\to\R^{d\times m}$ and an $m$-dimensional driving Brownian motion $W$, and assume that \eqref{sde0} has a unique strong solution $(X(t))_{t\in[0, T]}$. Our computational task is to approximate $X(T)$ by means of methods that use finitely many evaluations of the driving Brownian motion $W$. In particular we are interested in the following question: under which assumptions on the coefficients $\mu$ and $\sigma$ exists a method of the latter type, which converges to 
$X(T)$ in absolute mean with a polynomial rate?

It is well-known that if the coefficients $\mu$ and $\sigma$ are globally Lipschitz continuous then the classical Euler scheme achieves the rate of convergence $1/2$, see \cite{m55}. Moreover,
the recent literature on numerical approximation of SDEs contains a number
of results on approximation schemes that are specifically designed for non-Lipschitz coefficients and achieve
polynomial convergence rates for suitable classes of such SDEs, 
see e.g.
\cite{h96,hms02,
HutzenthalerJentzenKloeden2012, MaoSzpruch2013Rate,
WangGan2013,Sabanis2013ECP,TretyakovZhang2013,Beynetal2014,KumarSabanis2016, Beynetal2016}
for SDEs with globally monotone coefficients
and see e.g.
\cite{BerkaouiBossyDiop2008,GyoengyRasonyi2011,DereichNeuenkirchSzpruch2012,
Alfonsi2013,NeuenkirchSzpruch2014,HutzenthalerJentzen2014,
HutzenthalerJentzenNoll2014CIR, 
LS15b,LS16, Tag16, HH16b} 
for SDEs with possibly non-monotone coefficients.

On the other hand, it has recently been shown in \cite{JMGY15}  that for any sequence  $(a_n)_{n\in\N}\subset (0,\infty)$, which may converge to zero arbitrary slowly,  there exists an SDE \eqref{sde0} with $d=4$ and $m=1$ and with infinitely often differentiable and bounded coefficients $\mu$ and $\sigma$ such that no approximation of $X(T)$
based on finitely many
evaluations of the driving Brownian
motion $W$ converges in absolute mean  faster than the given sequence $(a_n)_{n\in\N}$. 
More formally, 
\begin{equation}
\label{eq:intro3} 
  \inf_{ s_1, \dots, s_n \in [0,T] }
  \inf_{
    \substack{
      u \colon \R^n \to \R^4
    \\
      \text{measurable}
    }
  }
  \EE
    \big\|
      X( T )
      -
      u\big( W( s_1 ), \dots, W( s_n )
      \big)
    \big\|
\geq  a_n.
\end{equation}
In particular, there exists an SDE \eqref{sde0} with infinitely often differentiable and bounded coefficients $\mu$ and $\sigma$ such that
its solution at the final time can not be approximated with a polynomial rate of convergence based on finitely many
evaluations of the driving Brownian
motion $W$. We add that the latter statement in the special case when the approximation is given by the Euler scheme with equidistant time steps has first been shown in \cite{hhj12}.

Note that the time points $s_1, \ldots, s_n\in[0,T]$ that are used by an approximation $u( W( s_1 ), \dots,$ $W( s_n ))$ in \eqref{eq:intro3} are fixed, and therefore
this negative result does not cover approximations that may choose the number as well as the
location of the evaluation sites of the driving Brownian motion $W$ in an
adaptive way, e.g. numerical schemes that adjust the actual step size according to a criterion that is based on the values of the driving Brownian motion $W$ observed so far, see e.g. \cite{Gaines1997,MG02_habil,m04, Moon2005, RW2006, LambaMattinglyStuart2007, Hoel2012,Hoel2014} and the references therein. See Section \ref{approximations} for the formal definition of that type of approximations.  It is
well-known that for SDEs \eqref{sde0} with (essentially) globally Lipschitz continuous coefficients $\mu$ and $\sigma$ adaptive approximations can not achieve a better rate of convergence compared to what is
best possible for non-adaptive ones, which at the same time coincides with the best possible rate of convergence that can be achieved by any approximation based on $W(\tfrac{T}{n}), W(\tfrac{2T}{n}), \ldots, W(T) $, see \cite{MG02_habil,m04}.
However, as has recently turned out, this is not necessarily the case anymore if the coefficients $\mu$ and $\sigma$ are not both globally Lipschitz continuous.  In ~\cite{HH16} it has been shown that for the one-dimensional squared Bessel process,
which is the solution of the SDE \eqref{sde0} with $d=m=\mu=1$ and $\sigma(x)=2\sqrt {|x|}$ for $x\in\R$ the following holds: the best possible rate of convergence that can be achieved by any approximation based on $W(\tfrac{T}{n}), W(\tfrac{2T}{n}), \ldots, W(T) $ equals $1/2$, i.e. there exist $c_1, c_2>0$ such that 
\[
 c_1\cdot n^{-1/2}\leq  \inf_{
    \substack{
      u \colon \R^n \to \R
    \\
      \text{measurable}
    }
  }
  \EE
    \big|
      X( T )
      -
      u\big( W(\tfrac{T}{n}), W(\tfrac{2T}{n}), \ldots, W(T)
      \big)
    \big|
\leq  c_2\cdot n^{-1/2},
\]
 while the best possible rate of convergence that can be achieved by  approximations based on $n$ adaptively chosen evaluations of the driving Brownian motion $W$ equals infinity. More formally, for every $\alpha>0$ there exists $c>0$ and a sequence of approximations $\widehat X_n$ based on $n$ adaptively chosen evaluations of $W$ such that 
\[
\EE |X(T)-\widehat X_n|\leq c\cdot
n^{-\alpha}.
\]

In view of the latter result one might hope that a non-polynomial lower error bound $a_n$
in \eqref{eq:intro3} could be overcome by using adaptive approximations, see also the discussion in ~\cite[p.\,2]{Gobet2016}. In the present article we prove that the pessimistic alternative is true. 
We show that for any sequence  $(a_n)_{n\in\N}\subset (0,\infty)$, which may converge to zero arbitrary slowly,  there exists an SDE \eqref{sde0} with $d=4$ and $m=1$ and with infinitely often differentiable and bounded coefficients $\mu$ and $\sigma$ such that no approximation 
based on $n$ adaptively chosen evaluations  of the driving Brownian
motion $W$ on average can achieve a smaller absolute mean error than the given number  $a_n$, i.e.
\[ 
  \EE
    \big\|
      X( T )
      -
     \widehat X_n
    \big\|
\geq  a_n
\]
for any approximation $ \widehat X_n$ of the latter type.
This fact is an immediate consequence of Corollary \ref{cor3} in Section \ref{results}  together with an appropriate scaling argument. 
For the proof of the latter result  we employ the same class of SDEs as in ~\cite{JMGY15}.
Thus, roughly speaking, these SDEs can not be solved approximately in the strong sense in a reasonable computational time by means of any kind of adaptive (or nonadaptive) method based on finitely many evaluations of the driving Brownian motion $W$.

We conjecture that a similar negative result does even if one  allows for adaptive  approximations based on finitely many  evaluations of arbitrary linear continuous functionals of the
driving Brownian motion $W$.
However, in this case one  can not employ the class of SDEs from \cite{JMGY15} since for every such SDE  its solution at the final time can be approximated with error zero based on the evaluation of only two linear continuous functionals of the
driving Brownian motion $W$, see \eqref{solutionSDE}

We add that negative results in the spirit of \eqref{eq:intro3} for quadrature problems for marginal distributions of SDEs have recently been established in \cite{MGY16}.

We briefly describe the content of the paper. In Section~\ref{not}
we fix some notation. 
In Section~\ref{sec:setting} we briefly introduce the class of SDEs from \cite{JMGY15}, which is studied in this article as well. In Section~\ref{approximations} we formally define the class 
of adaptive approximations, which are analysed in this article. Our lower error bounds are stated in Section \ref{results}. The proof of the
main result, Theorem \ref{t1}, is carried
out in Section~\ref{strong}.

\section{Notation}\label{not}

Throughout this article the following notation is used.
For a set $ A $, a vector space $ V$, a set $ B \subseteq V $,
and a function $ f \colon A \to B $
we
put
$
  \operatorname{supp}( f ) = \left\{ x \in A \colon f(x) \neq 0 \right\}
$. For sets $A$, $B$, a function $f \colon A \to B$ and a subset
$E\subseteq A$ we denote by $f|_E$ the restriction of $f$ to $E$.
Moreover, for $ d \in \N $ and $ v \in \R^d $ we write $
  \|v\|
$ for the Euclidean norm of $ v$. 
For $n\in\N$ and $-\infty<a<b<\infty$ we denote by  $\mathfrak
B(\R^n)$ and   $\mathfrak
B(C([a,b]))$  the Borel $\sigma$-fields on $\R^n$ and on $C([a,b])$, respectively, where the latter space is equipped with the supremum norm. For $S$ being a finite product of the latter two spaces we denote by  $\mathfrak
B(S)$ the Borel  $\sigma$-field on $S$ generated by the respective product topology.

\section{A family of SDEs with smooth and bounded coefficients}
\label{sec:setting}

Throughout this article we study 
SDEs 
provided by the following setting. 

Let
$ T  \in (0,\infty) $,
let
$ ( \Omega, \mathcal{F}, \PP ) $
be a probability space with a normal filtration
$ ( \mathcal{F}_t )_{ t \in [0,T] } $,
and let
$
  W \colon [0,T] \times \Omega \to \R
$
be a
standard $ ( \mathcal{F}_t )_{ t \in [0,T] } $-Brownian motion
on $ ( \Omega, \mathcal{F}, \PP ) $.

Let $ 0<\tau_1<\tau_2<T$ and let $ f, g, h \in C^{ \infty }(\R) $ be
bounded and satisfy $
  \operatorname{supp}( f ) \subseteq ( - \infty, \tau_1 ]
$,
$
  \inf_{ t\in [ 0, \tau_1/2  ] } | f'(t) | > 0
$,
$
  \operatorname{supp}( g ) \subseteq [ \tau_1, \tau_2 ]
$,
$g\not =0$,
$
  \operatorname{supp}( h ) \subseteq [ \tau_2, \infty )
$,
and
$
  \int_{ \tau_2 }^{ T } h(t) \, dt \neq 0
$.

For every $
  \psi \in C^{ \infty }( \R)
$
let
$
  \mu^{\psi} \colon \R^4 \to \R^4
$
and
$
  \sigma \colon \R^4 \to \R^4
$
be given by
\begin{align*}
  \mu^{ \psi }(x) & = \bigl( 1, 0, 0, h( x_1 ) \cdot \cos( x_2 \, \psi( x_3 ) ) \bigr),\\
  \sigma(x) &=
  \bigl(
    0, f( x_1 ) , g( x_1 ), 0
  \bigr)
\end{align*}
and
 consider the following $4$-dimensional system of SDEs
\begin{equation}\label{sde}
\begin{aligned}
dX^{\psi }(t) & = \mu^\psi (X^{\psi }(t)) \, dt + \sigma(X^{\psi }(t)) \, dW(t), \quad t\in [0,T],\\
X^{\psi}(0) & = 0.
\end{aligned}
\end{equation}

\begin{rem}
Note that for  every $ \psi \in C^{ \infty }( \R ) $ the
functions $ \mu^{ \psi } $ and $\sigma$ are infinitely often
differentiable and bounded.
\end{rem}

\begin{rem}\label{solution}
It is easy to see that for every $ \psi \in C^{ \infty }( \R
) $ the SDE \eqref{sde} has a unique strong solution given by
\begin{align}\label{solutionSDE}
  X_1^{ \psi }(t) & = t ,
\qquad
  X_2^{ \psi }(t) = \int_0^{ \min( t, \tau_1 ) } f(s) \, dW(s),\notag
\\
  X_3^{ \psi }(t) & = 1_{ [ \tau_1 , \, T ] }( t ) \cdot
  \int_{ \min( t, \tau_1 ) }^{ \min( t , \tau_2 ) } g(s) \, dW(s) ,
\\
  X_4^{ \psi }(t) & =
  1_{ [ \tau_2 , \, T ] }(t) \cdot
  \cos\bigl(
    X_2^{ \psi }( \tau_1 ) \,
    \psi\big(
      X_3^{ \psi }( \tau_2 )
    \big)
  \bigr)
  \cdot
  \int_{ \tau_2 }^{ t } h(s) \, ds\notag
\end{align}
for all $ t \in [0,T] $.
\end{rem}

\section{Adaptive strong approximations}
\label{approximations} Let $\delta\in (0,T]$. We study general
strong approximations of $X^\psi(T)$ based on $(W(t))_{t\in[\delta,
T]}$ and on finitely many sequential evaluations of $W$ in the
interval $(0, \delta)$. Every such approximation $\widehat X\colon
\Omega\to\R^4$ is defined by three sequences
\[
\pp=(\pp_n)_{n\in\N}, \quad \chi=(\chi_n)_{n\in\N},\quad
\phi=(\phi_n)_{n\in\N}
\]
of measurable mappings
\begin{align}\label{sequences}
\pp_n&\colon\R^{n-1} \times \R^{[\delta, T]} \to (0, \delta),\notag\\
\chi_n&\colon\R^{n}\times \R^{[\delta, T]}\to\{0,1\},\\
\phi_n&\colon\R^{n}\times \R^{[\delta, T]}\to\R^4.\notag
\end{align}
The sequence $\pp$ determines the evaluation sites of a trajectory
of $W$ in the interval $(0, \delta)$. The total number of
evaluations is determined by the sequence $\chi$ of stopping rules.
Finally, the sequence $\phi$ is used to obtain the approximation to
$X^\psi(T)$ from the observed data.

More precisely, let $\omega\in\Omega$, let $w=W(\omega)$ be the corresponding trajectory of $W$ and put $\ovw=(w(t))_{t\in[\delta, T]}$. The sequential observation of $w$
starts at the knot $\pp_1(\ovw)$. After $n$ steps the
available information is then given by $D_n(\omega)=(y_1,
\ldots, y_n, \ovw)$, where $y_1=w(\pp_1(\ovw))$, \ldots,
$y_n=w(\pp_n( y_1, \ldots, y_{n-1}, \ovw))$, and we decide
whether we stop or further evaluate $w$ according to the value of
$\chi_n(D_n(\omega))$. The total number of observations of $w$ in the
interval $(0, \delta)$ is thus given by
\begin{equation}\label{nu}
\nu(\omega)=\min\{n\in\N:\chi_n(D_n(\omega))=1\}.
\end{equation}
 If
$\nu(\omega)<\infty$, then the data $D_{\nu(\omega)}(\omega)$
is used to construct the estimate
$\phi_{\nu(\omega)}(D_{\nu(\omega)}(\omega))\in\R^4$.

For obvious reasons we require that $\nu<\infty$ $\PP$-a.s. Then
the resulting approximation is given by
\[
\widehat X=\phi_{\nu}(D_{\nu}).
\]
Without loss of generality we assume that 
\begin{equation}\label{15}
\pp_k(y_1, \ldots, y_{k-1}, \ovw)\not = \pp_l(y_1, \ldots, y_{l-1}, \ovw)
\end{equation}
for all $y\in\R^\infty$, all $k, l\in\N$ with $k\not=l$ and all $v\in C([\delta, T])$.
We put
\[
c(\widehat X)=\EE \nu,
\]
that is the expected number of evaluations of the driving Brownian
motion $W$ in the interval $(0, \delta)$. We denote by $\alg^\delta$
the class of all methods of the above form and for $N\in\N$ we put
\[
\alg_N^\delta=\{\widehat X\in\alg^\delta: c(\widehat X)\leq N\}.
\]
Clearly,
$
\alg_N^{\delta_1}\subseteq\alg_N^{\delta_2}
$
for all $0<\delta_2\leq \delta_1\leq T$ and all $N\in\N$.

Let us stress that the class $\alg_N^{\delta}$ contains in particular all methods from the literature, which use a step size control based on $N$ sequential 
evaluations of $W$ on average, see e.g. \cite{Gaines1997,MG02_habil,m04, Moon2005, RW2006, LambaMattinglyStuart2007, Hoel2012,Hoel2014} and the references therein. 
Moreover, $\alg_N^{\delta}$ of course contains
all nonadaptive approximations $u\bigl(W(s_1), \ldots, W(s_N), $ $(W(t))_{t\in[\delta, T]}\bigr)$ based on  $N$
evaluations of $W$ at fixed time points $s_1, \ldots, s_N\in(0, \delta)$ and on $(W(t))_{t\in[\delta, T]}$, as studied  in \cite{JMGY15}.
In the latter case one can take any sequences $\pp, \chi$ and $\phi$ satisfying
\[ \pp_n=s_n \text{ for } n\leq N, \chi_1=\ldots \chi_{N-1}=0, \chi_N=1 \text{ and } \phi_N=u.
\]

\section{Main results}
\label{results} Assume the setting in Section~\ref{sec:setting} and
put

\begin{equation}\label{51}
\alpha
  =
  \inf_{ t \in [ 0 , \tau_1 /2  ] }
  | f'(t) |^2
  ,\quad
  \beta
   =
   \int_{\tau_1}^{\tau_2}  g^2(t) dt,\quad
   \gamma
   =
  \int_{ \tau_2 }^T
  h(t) \, dt
\end{equation}
as well as
\[
c_1=\frac{ \gamma \exp \bigl(- \tfrac{\pi^2}{4}-\tfrac{1}{\beta}\bigr)}{  8 \pi\sqrt{2\pi\beta}  },\quad c_2=\frac{ \gamma \exp \bigl(- \tfrac{\pi^2}{4}\bigr)}{  4 \pi  }.
\]

Our main result is stated in Theorem \ref{t1}. It provides a uniform lower bound for the mean absolute error of any
strong approximation of $ X^{ \psi }(T) $ that is based on $(W(t))_{t\in [\delta, T]}$ and on $N$
sequential evaluations of $ W $ in the
interval $(0, \delta)$ on average in the case that $\psi$ is positive, strictly increasing and satisfies $
  \lim_{ x \to \infty } \psi( x ) = \infty
$ as well as
$1\in\psi(\R)$. See Section \ref{strong} for the proof.
\begin{theorem}
\label{t1}  Let $
  \delta\in(0, T]
$ and let $ \psi \in C^{ \infty }( \R) $ be positive, strictly
increasing  with $
  \lim_{ x \to \infty } \psi( x ) = \infty
$ and 
$1\in\psi(\R)$. Then for all $N\in\N$ and all $\widehat X\in\alg_N^\delta$ we
have 
\begin{equation}\label{61}
  \EE
    \|
      X^{ \psi }( T ) -\widehat X
    \|\geq c_1\cdot \exp\bigl(-\tfrac{1}{\beta}\cdot \bigl(\psi^{-1}\bigl(\bigl(1+ \sqrt{\tfrac{96 }{\alpha (\min(\delta, \tau_1/2))^3}}\bigr)  N^3\bigr)\bigr)^2\bigr)- \frac{c_2}{N}.
\end{equation}

\end{theorem}

As a consequence of Theorem \ref{t1} we obtain a non-polynomial decay of the smallest possible mean absolute error of
strong approximation of $ X^{ \psi }(T) $ based on $(W(t))_{t\in [\delta, T]}$ and on $N$
sequential evaluations of $ W $ in the
interval $(0, \delta)$ on average if $\psi$ additionally satisfies an exponential growth condition.

\begin{cor}\label{cor1} Let $
  \delta\in(0, T]
$ and let $ \psi \in C^{ \infty }( \R) $ be positive, strictly
increasing  with $
  \lim_{ x \to \infty } \psi( x ) = \infty
$ and 
$1\in\psi(\R)$. Moreover assume that for all $q\in (0,\infty)$
\[
\lim_{x\to\infty} \psi(x)\cdot \exp(-q x^2)=\infty.
\]
Then for all $q\in(0, \infty)$ we have
\[
\lim_{N\to\infty} \bigl(N^q\cdot \inf_{\widehat X\in \alg^\delta_N}  \EE
    \|
      X^{ \psi }( T ) -\widehat X
    \|\bigr) =\infty.
\]

\end{cor}
\begin{proof}
The assumptions on the function $\psi$ ensure that for all $q\in(0, \infty)$
\begin{equation}\label{78}
\lim_{N\to\infty}      \bigl(N^q\cdot \exp\bigl(-\tfrac{1}{\beta}\cdot \bigl(\psi^{-1}\bigl(\bigl(1+ \sqrt{\tfrac{96 }{\alpha (\min(\delta, \tau_1/2))^3}}\bigr)  N^3\bigr)\bigr)^2\bigr)\bigr)=\infty,
\end{equation}
see Lemma 4.5 in \cite{JMGY15}. This in particular implies that there exists $N_0\in\N$ such that for all $N\geq N_0$ 
\[
\frac{c_1}{2}\cdot \exp\bigl(-\tfrac{1}{\beta}\cdot \bigl(\psi^{-1}\bigl(\bigl(1+ \sqrt{\tfrac{96 }{\alpha (\min(\delta, \tau_1/2))^3}}\bigr)  N^3\bigr)\bigr)^2\bigr)\geq \frac{c_2}{N}.
\]
Employing Theorem \ref{t1}  we therefore conclude that for all $N\geq N_0$ 
\[
  \inf_{\widehat X\in \alg^\delta_N} \EE
    \|
      X^{ \psi }( T ) -\widehat X
    \|\geq \frac{c_1}{2}\cdot \exp\bigl(-\tfrac{1}{\beta}\cdot \bigl(\psi^{-1}\bigl(\bigl(1+ \sqrt{\tfrac{96 }{\alpha (\min(\delta, \tau_1/2))^3}}\bigr)  N^3\bigr)\bigr)^2\bigr).
\]
The latter estimate and  \eqref{78} imply the statement of the corollary.
\end{proof}

The following result shows that the smallest possible mean absolute error of
strong approximation of $X^\psi(T)$ based  $(W(t))_{t\in [\delta, T]}$ and on $N$
sequential evaluations of $ W $ in the
interval $(0, \delta_N)$ on average may converge to zero  arbitrarily slow even then when  the sequence $(\delta_N)_{N\in\N}$ tends to zero 
with any given speed.

\begin{cor}\label{cor3} 
Let $(a_N)_{N\in\N}\subset (0, \infty)$ and $(\delta_N)_{N\in\N}\subset (0, T]$ satisfy $\lim_{N\to\infty} a_N=0$ and $\lim_{N\to\infty} \delta_N=0$. Then there exists $\kappa>0$ and $ \psi \in C^{ \infty }( \R) $ such that for all $N\in\N$ we
have 
\[
 \inf_{\widehat X\in \alg^{\delta_N}_N} \EE
    \|
      X^{ \psi }( T ) -\widehat X
    \|\geq \kappa\cdot a_N.
\]
\end{cor}
\begin{proof} 
We proceed similar to the proof of Corollary 4.3 in \cite{JMGY15}.
Without loss of generality we may assume that the sequences $
  ( a_N )_{ N \in \N }
$ and $
  ( \delta_N )_{ N \in \N }
$ are strictly decreasing.
Let
\[
N_0=\min\bigl\{N\in\N\colon a_N+\frac{c_2}{N}\leq c_1\bigr\}
\]  
and for $N\geq  N_0$ put 
\[
  b_N =\sqrt{
  -\beta
    \ln\!\big(
      \tfrac{ 1}{c_1}\cdot\big(a_N+\tfrac{ c_2 }{N}
    \big)\big)}, \quad   d_N =
  \bigl(1+ \sqrt{\tfrac{96 }{\alpha (\min(\delta_N, \tau_1/2))^3}}\bigr) N^3
  .
\]
Note that the sequences $
  ( b_N )_{ N\geq N_0 }
$ and $
  ( d_N )_{ N\geq N_0 }
$ are strictly increasing and satisfy 
\[
 \lim_{ N \to \infty } b_N =  \lim_{ N \to \infty } d_N = \infty
.
\]
Define a function $
  \psi \colon \R \to \R
$ by
\[
  \psi(x) =
  \begin{cases}
   d_{N_0}\cdot\bigl( 1 -
    \exp\!\big(
      \frac{ 1 }{
        x - b_{N_0} 
      }
    \big)\bigr)
    ,
  &
    \text{if }
    x < b_{N_0}
    ,
\\[1ex]
  d_N ,
  &
  \text{if } x = b_N \text{ and } N\ge N_0,
\\[1ex]
 d_{N-1}  +
  \displaystyle{
    \frac{
       d_N  -  d_{N-1}
    }{
              1 +
        \exp\!\big(
          \frac{ 1 }{  x - b_{ N - 1 }  }
          -
          \frac{ 1 }{
             b_N - x 
          }
        \big)
          }
  }
  ,
  &
  \text{if }
  x \in ( b_{ N - 1 } , b_N ) \text{ and } N > N_0
  .
\end{cases}
\]
Then $ \psi $ is positive, strictly increasing,  infinitely often
differentiable and satisfies\\  $
  \lim_{ x \to \infty } \psi( x ) = \infty
$ as well as $1\in\psi( \R )$.

For $N\in\N$ put
\[
\varepsilon_N =  \inf_{\widehat X\in\alg_N^{\delta_N}} \EE
    \|
      X^{ \psi }( T ) -\widehat X
    \|.
\]
Theorem \ref{t1} implies that for
all $ N\geq N_0$ 
\[
  \varepsilon_N  \geq c_1
  \cdot
  \exp\bigl(
    - \tfrac{ 1 }{ \beta }
    \cdot
    (
      \psi^{ - 1 }(
       d_N
      )
    )^2
  \bigr)-\frac{c_2}{N}
  = c_1  \cdot\exp\bigl( - \tfrac{ 1 }{ \beta } \cdot b_N ^2 \bigr)-\frac{c_2}{N} = a_N
  .
\]
Since the sequence $(\varepsilon_N)_{N\in\N}$ is decreasing hence for all $ N \in \{ 1, 2, \dots, N_0 \} $ 
\[
 \varepsilon_N
\ge \varepsilon_{N_0} \ge a_{ N_0 }.
\]
Using the assumption that the sequence $
  ( a_N )_{ N \in \N }
$ is strictly decreasing we therefore conclude that for all $N\in\N$ 
\[
\varepsilon_N \ge  \min\{1,a_{N_0}/a_N\}\cdot  a_N \ge
\frac{a_{N_0}}{a_1}\cdot a_N,
\]
which completes the proof of the corollary with $\kappa= 
a_{N_0}/a_1$.
\end{proof}

\section{Proof of Theorem \ref{t1}}
\label{strong}

Let $\delta\in (0, T]$ and let $\widehat X\in\alg^\delta$ be given
by sequences $\pp=(\pp_n)_{n\in\N}, \chi=(\chi_n)_{n\in\N}$ and
$\phi=(\phi_n)_{n\in\N}$, see \eqref{sequences}. Recall the
definition \eqref{nu} of $\nu$. We first determine the regular
conditional distribution $\PP^{W|D_\nu}$.

For $n\in\N$ put
\[
S_n=\{s\in(0, \delta)^n: |\{s_1, \ldots, s_n\}|=n\}.
\] For $n\in\N$, $s\in S_n$, $y\in\R^n$ and $v\in
C([\delta, T])$  define functions
\[
m_{s,y,v}\colon [0, T]\to \R\,\,\text{ and }\,\, R_s\colon [0,
T]^2\to \R
\] as follows.
If $s_1<\ldots <s_n$  put $s_0=y_0=0$, $s_{n+1}=\delta$ and
$y_{n+1}=v(\delta)$ and let
\[
m_{s,y,v}(t)=\begin{cases}
 \frac{s_{k}-t}{s_k-s_{k-1}}\cdot y_{k-1}+\frac{t-s_{k-1}}{s_k-s_{k-1}}\cdot y_k
    ,
  &
    \text{if }
    t\in [s_{k-1}, s_k) \text{ for } k\in\{1, \ldots, n+1\},  \\
    v(t)
    ,
  &
    \text{if }
    t\in [\delta, T]
\end{cases}
\]
as well as
\[
R_s(r,t)=\begin{cases}
 \frac{(s_{k}-\max(r,t))\cdot (\min(r,t)-s_{k-1})}{s_k-s_{k-1}}
    ,
  &
    \text{if }
    r,t\in [s_{k-1}, s_k) \text{ for } k\in\{1, \ldots, n+1\},  \\
    0
    ,
  &
    \text{otherwise},
\end{cases}
\]
for $r,t\in [0,T]$. Otherwise put
\[
m_{s,y,v}=m_{(s_{\pi(1)},\ldots,s_{\pi(n)}),
(y_{\pi(1)},\ldots,y_{\pi(n)}),v}, \quad
R_s=R_{(s_{\pi(1)},\ldots,s_{\pi(n)})},
\]
where $\pi$ is the permutation of $\{1, \ldots, n\}$ such that $s_{\pi(1)}<\ldots<s_{\pi(n)}$.

 For $n\in\N$, $y\in\R^n$,
$v\in C([\delta, T])$ and $k=1,\ldots,n$ put
\begin{equation}\label{14}
s^{y,v}_k=\pp_k( y_1, \ldots, y_{k-1},v).
\end{equation}
Note that due to the assumption \eqref{15} we have $|\{s^{y,v}_1,
\ldots, s^{y,v}_n\}|=n$. Let $Q_{y,v}$ denote the Gaussian measure
on $\mathfrak B(C([0, T]))$ with mean $m(Q_{y,v})=m_{s^{y,v},y,v}$
and covariance function $R(Q_{y,v})=R_{s^{y,v}}$. Consider the
measurable space
\begin{equation}\label{Omega1}
 (\Omega_1, \mathcal F_1)=\Bigl(\bigcup_{n=1}^\infty \R^n\times C([\delta, T]), \sigma\Bigl(\bigcup_{n=1}^\infty \mathfrak B(\R^n\times C([\delta, T]))\Bigr)\Bigr).
\end{equation}
It  is easy to see that $D_\nu\colon\Omega\to\Omega_1$ is $\mathcal
F$-$\mathcal F_1$ measurable. Define the mapping
\[
K\colon \Omega_1 \times \mathfrak B(C([0, T]))\to [0,1]
\]
by
\[
K((y,v), A)=Q_{y,v}(A)
\]
for all $y\in{\displaystyle \bigcup_{n=1}^\infty \R^n}, v\in
C([\delta, T])$  and $A\in \mathfrak B(C([0, T]))$.

\begin{lemma}\label{regular}
$K$ is a version of the regular conditional distribution
$\PP^{W|D_\nu}$.
\end{lemma}

In the case of $\delta=T$ the statement of Lemma \ref{regular}
seems to be well-known, see, e.g., \cite{hmr01, m02, MG02_habil, m04}, but a
proof of it seems not to be available in the literature. If,
additionally, $\nu$ is constant then Lemma \ref{regular} follows
from Lemma 2.9.7 in~\cite[p.~474]{TWW88}, but measurability issues
have not been fully addressed in the proof of the latter result. For
convenience of the reader we therefore provide a proof of Lemma
\ref{regular} here.

\begin{proof}
Clearly, for all $(y,v)\in \Omega_1$ the mapping
\[
\mathfrak B(C([0, T]))\ni A\mapsto K((y,v), A)\in [0,1]
\]
is a probability measure on $\mathfrak B(C([0, T]))$.

Next, let $A\in\mathfrak B(C([0, T]))$. We show that the mapping
\begin{equation}\label{l6}
\Omega_1\ni (y,v) \mapsto K((y,v), A)\in [0,1]
\end{equation}
is $\mathcal F_1$\,-\,$\mathfrak B([0,1])$ measurable. For $n\in\N$,
$s\in S_n$ and $u\in C([0,\delta])$ define a function
\[
F_{s,u}\colon [0,T]\to\R
\]
by
\[
F_{s,u}(t)=\begin{cases}
 u(t)-m_{s, (u(s_1), \ldots, u(s_n)),0}(t)
    ,
  &
    \text{if }
    t\in [0,\delta],
    \\
    0
    ,
  &
    \text{if }
    t\in (\delta, T]
    \end{cases}
\]
for $t\in[0, T]$.
 It is
easy to see that for all $(y, v)\in \Omega_1$
\[
Q_{y,v}=\PP^{F_{s^{y,v},(W(t))_{t\in[0, \delta]}}+m_{s^{y,v},y,v}},
\]
and therefore for all $(y, v)\in \Omega_1$
\begin{equation}\label{l5}
 K((y,v), A)=\int_{C([0, \delta])} \1_A(F_{s^{y,v},u}+m_{s^{y,v},y,v})\,\,\PP^{(W(t))_{t\in[0, \delta]}}(du).
\end{equation}
Clearly,
\[
\{((y,v),u)\in\Omega_1\times C([0, \delta]):
F_{s^{y,v},u}+m_{s^{y,v},y,v}\in A\}=\bigcup_{n=1}^\infty A_n,
\]
where
\[
A_n=\{(y,v,u)\in\R^n\times C([\delta,T])\times C([0, \delta]):
F_{s^{y,v},u}+m_{s^{y,v},y,v}\in A\}
\]
 for $n\in\N$.
The measurability of the functions  $\pp_n$, $n\in\N$, imply that for
every $n\in\N$ the mapping
\[
\R^n\times C([\delta, T]) \ni (y,v)\mapsto s^{y,v}\in S_n
\]
is $\mathfrak{B}(\R^n\times C([\delta, T]))$\,-\,$\mathfrak{B}(S_n)$
measurable. Thus, observing that for every $n\in\N$ the mappings
\[
S_n\times C([0,\delta]) \ni (s,u)\mapsto F_{s,u}\in C([0, T])
\]
and
\[
S_n\times\R^n\times C([\delta, T]) \ni (s,y,v)\mapsto m_{s,y,v}\in C([0, T])
\]
are continuous we conclude that for every  $n\in\N$ the mapping
\[
\R^n\times C([\delta, T])\times C([0,\delta]) \ni (y,v,u)\mapsto
F_{s^{y,v},u}+m_{s^{y,v},y,v}\in C([0, T])
\]
is $\mathfrak{B}(\R^n\times C([\delta, T])\times
C([0,\delta])$\,-\,$\mathfrak B(C([0, T]))$ measurable. Hence for
every  $n\in\N$
\[
A_n\in \mathfrak{B}(\R^n\times C([\delta, T])\times
C([0,\delta]))\subset \mathcal F_1\otimes \mathfrak
B(C([0,\delta])),
\]
 which implies that the mapping
\[
\Omega_1\times C([0, \delta]) \ni((y,v),u)\mapsto
\1_A(F_{s^{y,v},u}+m_{s^{y,v},y,v})\in\R
\]
is $\mathcal F_1\otimes \mathfrak B(C([0,\delta]))$\,-\,$\mathfrak
B(R)$ measurable. Using \eqref{l5} and employing Fubini's theorem we
thus conclude that the mapping \eqref{l6} is $\mathcal
F_1$\,-\,$\mathfrak B([0,1])$ measurable.

Finally, let $A\in\mathfrak B(C([0, T]))$ and $E\in \mathcal F_1$.
We show that
\begin{equation}\label{l3}
\PP (\{W\in A\}\cap\{D_\nu\in E\})=\int_E K((y,v), A)\,\,
\PP^{D_\nu}(d(y,v)).
\end{equation}
We have
\begin{align*}
\{D_\nu\in E\}\cap \{\nu<\infty\}&=\bigcup_{n=1}^\infty (\{D_\nu\in
E\}\cap \{\nu=n\})\\
&=\bigcup_{n=1}^\infty\Bigl(\{D_n\in E\}\cap
\{\chi_n(D_n)=1\}\cap\bigcap_{k=1}^{n-1}\{\chi_k(D_k)=0\}\Bigr)\\
&=\bigcup_{n=1}^\infty \{D_n\in E\cap C_n\},
\end{align*}
where $C_n\in \mathfrak B(\R^n\times C([\delta, T]))$ is given by
\[
C_n= \chi_n^{-1}(\{1\})\cap \bigcap_{k=1}^{n-1} \{(y,v)\in
\R^n\times C([\delta, T]): (y_1, \ldots, y_k,v)\in
\chi_k^{-1}(\{0\})\}
\]
for $n\in\N$. Since $\nu<\infty$ a.s. we thus obtain
\begin{align*}
\PP (\{W\in A\}\cap\{D_\nu\in E\})&=\sum_{n=1}^\infty \PP (\{W\in
A\}\cap\{D_n\in E\cap C_n\})\\
&= \sum_{n=1}^\infty \int_{E\cap C_n}\PP^{W|D_n=(y,v)} (A)
\,\,\PP^{D_n}(d(y,v)).
\end{align*}
Similarly to the proof of Lemma 2.9.7 on page 474 in ~\cite{TWW88}
one can show that for all $n\in\N$  and all $G\in\mathfrak
B(\R^n\times C([\delta, T]))$
\[
\int_{G}\PP^{W|D_n=(y,v)} (A) \,\,\PP^{D_n}(d(y,v))=\int_{G}Q_{y,v}
(A) \,\,\PP^{D_n}(d(y,v)).
\]
Hence
\begin{align*}
\PP (\{W\in A\}\cap\{D_\nu\in E\})&= \sum_{n=1}^\infty \int_{E\cap C_n}Q_{y,v}(A)
\,\,\PP^{D_n}(d(y,v))\\
&=\sum_{n=1}^\infty \int_{E\cap C_n}K((y,v), A)
\,\,\PP^{D_n}(d(y,v)).
\end{align*}
Since $\PP^{D_n}(G\cap C_n) = \PP^{D_\nu} (G)$  for all $n\in\N$ and
all $G\in\mathfrak B(\R^n\times C([\delta, T]))$ we thus conclude
\[
\PP (\{W\in A\}\cap\{D_\nu\in E\})= \sum_{n=1}^\infty \int_{E\cap (\R^n\times C([\delta, T]))}K((y,v), A)
\,\,\PP^{D_\nu}(d(y,v)),
\]
 which implies \eqref{l3} and completes the proof of the
lemma.
\end{proof}

Next, assume that $\delta\in(0, \tau_1]$, let $n\in\N$,  $y\in\R^n$ and $v\in C([\delta,T])$. Recall the definition \eqref{14} of the time points
$s^{y,v}_1, \ldots, s^{y,v}_n$ and put 
\begin{equation}\label{36}
s^{y,v}_0=0,\quad s^{y,v}_{n+1}=\delta.
\end{equation}
Let $\pi$ be the permutation of $\{0, \ldots, n+1\}$ such that $s^{y,v}_{\pi(0)}<\ldots <s^{y,v}_{\pi(n+1)}$. 
  Let $i^*\in\{0, \ldots, n\}$ and put
\begin{equation}\label{t0t1}
t_0=s^{y,v}_{\pi(i^*)}, \quad t_1=s^{y,v}_{\pi(i^*+1)}.
\end{equation}
Define mappings
\[
\widetilde W\colon C([0,T])\to C([0, t_0]\cup [t_1, \tau_1]), \,\, B\colon C([0,T])\to C([t_0, t_1])
\]
by
\begin{equation}\label{18}
\widetilde W(w)=w|_{[0, t_0]\cup [t_1, \tau_1]}
\end{equation}
for $w\in C([0,T])$ and
\begin{equation}\label{20}
B(w)(t)=w(t)-\frac{t_1-t}{t_1-t_0}w(t_0)-\frac{t-t_0}{t_1-t_0}w(t_1)
\end{equation}
for $w\in C([0,T])$ and $t\in [t_0, t_1]$. In the following lemma we present properties of the measures $Q_{y,v}^{\widetilde W}$, $Q_{y,v}^{ B}$ and $Q_{y,v}^{(\widetilde W, B)}$, which will be used in the proof of Theorem~\ref{t1}.
\begin{lemma}\label{induced}
We have
\begin{itemize}
\item[(i)] $Q_{y,v}^{\widetilde W}$ is the Gaussian measure on $\mathfrak B(C([0, t_0]\cup [t_1, \tau_1]))$ with  mean\\ $m(Q_{y,v}^{\widetilde W})=m_{s^{y,v},y,v}|_{[0, t_0]\cup [t_1, \tau_1]}$ and covariance function
$R(Q_{y,v}^{\widetilde W})=R_{s^{y,v}}|_{([0, t_0]\cup [t_1,
\tau_1])^2}$.
\item[(ii)]$Q_{y,v}^{B}$ is the Gaussian measure on $\mathfrak B(C([t_0, t_1]))$ with  mean $m(Q_{y,v}^{B})=0$ and covariance function
\[
R(Q_{y,v}^{B})(r,t)=\frac{(t_1-\max(r,t))\cdot (\min(r,t)-t_0)}{t_1-t_0}, \quad r,t\in[t_0, t_1].
\]
\item[(iii)] $Q_{y,v}^B=Q_{y,v}^{-B}$,
\item[(iv)] $Q_{y,v}^{(\widetilde W, B)}=Q_{y,v}^{\widetilde W}\times Q_{y,v}^{B}$.
\end{itemize}
\end{lemma}
\begin{proof}
The property (i) is obvious. Since $B$ is a linear continuous mapping hence $Q_{y,v}^{B}$ is a
Gaussian measure. Next, observe that for all $t\in [t_0, t_1]$ we have
\begin{align*}
m(Q_{y,v})(t)&=\frac{t_1-t}{t_1-t_0}y_{\pi(i^*)}+\frac{t-t_0}{t_1-t_0}y_{\pi(i^*+1)}.
\end{align*}
In particular, for all $t\in[t_0, t_1]$,
\begin{equation}\label{19}
m(Q_{y,v})(t)-\frac{t_1-t}{t_1-t_0}m(Q_{y,v})(t_0)
-\frac{t-t_0}{t_1-t_0}m(Q_{y,v})(t_1)=0.
\end{equation}
Hence for all
$t\in [t_0, t_1]$,
\begin{align*}
m(Q_{y,v}^{B})(t)&=\int_{C([t_0, t_1])}b(t)Q_{y,v}^{B}(db)\\
&= \int_{C([0, T])}\Bigl(w(t)-\frac{t_1-t}{t_1-t_0}w(t_0)-\frac{t-t_0}{t_1-t_0}w(t_1)\Bigr)Q_{y,v}(dw)\\
&=m(Q_{y,v})(t)-\frac{t_1-t}{t_1-t_0}m(Q_{y,v})(t_0) -\frac{t-t_0}{t_1-t_0}m(Q_{y,v})(t_1)=0
\end{align*}
and for all $r,t\in [t_0, t_1]$,
\begin{align*}
R(Q_{y,v}^{B})(r,t)
&=\int_{C([t_0, t_1])}\bigl(b(r)-m(Q_{y,v}^{B})(r)\bigr)\cdot \bigl(b(t)-m(Q_{y,v}^{B})(t)\bigr) Q_{y,v}^{B}(db)\\
&= \int_{C([0, T])}\Bigl(w(r)-\frac{t_1-r}{t_1-t_0}w(t_0)-\frac{r-t_0}{t_1-t_0}w(t_1)\Bigr)\\
&\hspace{3cm}\cdot \Bigl(w(t)-\frac{t_1-t}{t_1-t_0}w(t_0)-\frac{t-t_0}{t_1-t_0}w(t_1)\Bigr)Q_{y,v}(dw),
\end{align*}
and hence
\begin{align*}
R(Q_{y,v}^{B})(r,t)&= \int_{C([0, T])}\Bigl(w(r)-m(Q_{y,v})(r)-\frac{t_1-r}{t_1-t_0}(w(t_0)-m(Q_{y,v})(t_0))\\
&\hspace{6cm}-\frac{r-t_0}{t_1-t_0}(w(t_1)-m(Q_{y,v})(t_1))\Bigr)\\
&\hspace{3cm}\cdot \Bigl(w(t)-m(Q_{y,v})(t)-\frac{t_1-t}{t_1-t_0}(w(t_0)-m(Q_{y,v})(t_0))\\
&\hspace{6cm}-\frac{t-t_0}{t_1-t_0}(w(t_1)-m(Q_{y,v})(t_1))\Bigr)Q_{y,v}(dw)\\
&=R(Q_{y,v})(r,t)-\frac{t_1-t}{t_1-t_0}R(Q_{y,v})(r,t_0)-\frac{t-t_0}{t_1-t_0}R(Q_{y,v})(r,t_1)\\
&\quad-\frac{t_1-r}{t_1-t_0}\Bigl(R(Q_{y,v})(t_0,t)-\frac{t_1-t}{t_1-t_0}R(Q_{y,v})(t_0,t_0)-\frac{t-t_0}{t_1-t_0}R(Q_{y,v})(t_0,t_1)\Bigr)\\
&\quad-\frac{r-t_0}{t_1-t_0}\Bigl(R(Q_{y,v})(t_1,t)-\frac{t_1-t}{t_1-t_0}R(Q_{y,v})(t_1,t_0)-\frac{t-t_0}{t_1-t_0}R(Q_{y,v})(t_1,t_1)\Bigr).
\end{align*}
Observing that $R(Q_{y,v})(t_0,t)=R(Q_{y,v})(t_1,t)=0$ for all $t\in [t_0, t_1]$ we thus obtain
\[
R(Q_{y,v}^{B})(r,t)=R(Q_{y,v})(r,t)
\]
for all $r,t\in [t_0, t_1]$, which completes the proof of (ii). The property (ii) implies that
$Q_{y,v}^{-B}$ is the Gaussian measure on $\mathfrak B(C([t_0, t_1]))$ with  mean $m(Q_{y,v}^{-B})=0$ and
covariance function $R(Q_{y,v}^{-B})=R(Q_{y,v}^{B})$, which yields
the property (iii). Next, we prove (iv). Using the properties (i),
(ii) as well as \eqref{19} and the fact that $R(Q_{y,v})(r,t)=0$ for all $r\in [0,t_0]\cup [t_1,
\tau_1]$ and $t\in [t_0, t_1]$ we obtain
\begin{align*}
&\int_{C([0,T])}\bigl(\widetilde W(w)(r)-m(Q_{y,v}^{\widetilde
W})(r)\bigr)\cdot\bigl(B(w)(t)-m(Q_{y,v}^{B})(t)\bigr)
Q_{y,v}(dw)\\
&=\int_{C([0,T])}\bigl(w(r)-m(Q_{y,v})(r)\bigr)\cdot\Bigl(w(t)-\frac{t_1-t}{t_1-t_0}w(t_0)-\frac{t-t_0}{t_1-t_0}w(t_1)\Bigr)
Q_{y,v}(dw)\\
&=\int_{C([0,T])}\bigl(w(r)-m(Q_{y,v})(r)\bigr)\cdot\Bigl(w(t)-m(Q_{y,v})(t)-\frac{t_1-t}{t_1-t_0}(w(t_0)-m(Q_{y,v})(t_0))\\
&\hspace{8cm}-\frac{t-t_0}{t_1-t_0}(w(t_1)-m(Q_{y,v})(t_1))\Bigr)
Q_{y,v}(dw)\\
&=R(Q_{y,v})(r,t)-\frac{t_1-t}{t_1-t_0}R(Q_{y,v})(r,t_0)-\frac{t-t_0}{t_1-t_0}
R(Q_{y,v})(r,t_1)=0
\end{align*} for all $r\in [0,t_0]\cup [t_1,
\tau_1]$ and $t\in [t_0, t_1]$, which means that $\widetilde W(r)$ and $B(t)$ are uncorrelated.
This and the fact that $Q_{y,v}^{\widetilde W}$ and $Q_{y,v}^{B}$
are Gaussian measures implies (iv).


\end{proof}

In the proof of Theorem~\ref{t1} we employ the following lower bound
for the first absolute moment of the sine of a normally distributed
random variable, which is a generalization of Lemma 4.2 from
~\cite{JMGY15}, where a centered normally distributed random
variable has been considered.

\begin{lemma}
\label{l2}
Let $a\in\R$, $ \tau \in [1,\infty) $,
and let $ Y \colon \Omega \to \R $ be a
$ \mathcal{N}( a, \tau^2 ) $-distributed random variable.
Then
\[
  \EE\big[
    | \sin(Y) |
  \big]
  \ge
  \frac{ \exp\bigl(- \tfrac{\pi^2}{8}\bigr)}{ \sqrt{ 8 \pi } }
  .
\]
\end{lemma}
The proof of Lemma \ref{l2} is a straighforward generalization of the proof of Lemma 4.2 from
~\cite{JMGY15}.

We now proceed with the proof of Theorem \ref{t1}.

 We first consider the case  $\delta\in (0, \tau_1/2]$.
Let $N\in\N$ be such that $\widehat X\in \alg_{N}^\delta$. We may then assume that $\widehat X\in \alg_{N+1}^\delta$ and that $\widehat X$ uses  the
evaluation site $\delta/2$. 
  Let
$ \psi \in C^{ \infty }( \R) $ be positive, strictly increasing
with $
  \lim_{ x \to \infty } \psi( x ) = \infty
$ and $
  1\in \psi(\R).
$
Using Remark \ref{solution}, the assumptions on the functions $f$
and $g$ and It\^o's formula we obtain that $\PP$-a.s.
\begin{equation}\label{44}
X_2^\psi(\tau_1)=-\int_0^{\tau_1} f'(t)W(t)dt, \quad X_3^\psi(\tau_2)=-\int_{\tau_1}^{\tau_2} g'(t)W(t)dt.
\end{equation}
Thus, $\PP$-a.s.
\begin{equation}\label{11}
X_4^{ \psi }( T )=F(W),
\end{equation}
where $F\colon C([0,T])\to \R$ is given by
\[
F(w)=\gamma \cdot \cos\Bigl(\int_0^{\tau_1} f'(t)w(t)dt \cdot \psi\Bigl(-\int_{\tau_1}^{\tau_2} g'(t)w(t)dt\Bigr)\Bigr)
\]
and $\gamma$ is defined in \eqref{51}.
Recall the definition \eqref{Omega1} of the measurable space $(\Omega_1, \mathcal F_1)$ and define a function $G\colon \Omega_1\to\R^4$ by
\[
G(y,v)=\phi_n(y,v)
\]
for $y\in\R^n$ and $v\in C([\delta, T])$. Due to the measurability of the functions $\phi_n$, $n\in\N$, the function $G$ is $\mathcal F_1$-$\mathfrak B(\R^4)$ measurable. Moreover,
\begin{equation}\label{12}
\widehat X=\phi_{\nu}(D_{\nu})=G(D_{\nu}).
\end{equation}
Let $\text{pr}_4\colon \R^4\to\R$ denote the projection to the
fourth component. Due to \eqref{11} and \eqref{12} we have
\begin{equation}\label{42}
\EE| X_4^{ \psi }( T ) -\widehat X_4|=\EE
    |
      F(W) -\text{pr}_4(G(D_{\nu}))
    |=\EE\,\bigl(\EE\bigl(
    |
      F(W) -\text{pr}_4(G(D_{\nu}))
    |\,\bigr|\,D_\nu\bigr)\bigr).
\end{equation}
Lemma \ref{regular} implies that for $\PP^{D_\nu}$-a.a. $(y,v)\in\Omega_1$
\begin{align}\label{13}
\EE\bigl(
    |
      F(W) -\text{pr}_4(G(D_{\nu}))
    |\,\bigr|\,D_\nu=(y,v)\bigr)
    &=\int_{C([0,T])}  |
      F(w) -\text{pr}_4(G(y,v))
    |\, Q_{y,v}(dw).
\end{align}

Fix $n\in\N$,  $y\in\R^n$ and $v\in C([\delta,T])$. 
We show that
\begin{align}\label{41}
&\int_{C([0,T])}  |
      F(w) -\text{pr}_4(G(y,v))
    |\, Q_{y,v}(dw)\notag\\
    &\qquad\qquad\qquad\geq \frac{ \gamma \exp \bigl( - \tfrac{\pi^2}{ 4}\bigr) }{  8 \pi  }
  \cdot 1_{[1, \infty)}\Bigl(\sqrt{\tfrac{\alpha\delta^3}{96 n^3}}\cdot\psi\Bigl(-\int_{\tau_1}^{\tau_2} g'(t)\, v(t)dt\Bigr)\Bigr),
\end{align}
where $\alpha$ is given by \eqref{51}.
Recall the definitions \eqref{14} and \eqref{36} of the time points
$s^{y,v}_0, \ldots, s^{y,v}_{n+1}$, let $\pi$ be the permutation of $\{0, \ldots, n+1\}$ such that $s^{y,v}_{\pi(0)}<\ldots <s^{y,v}_{\pi(n+1)}$,  and  
let $t_0$ and  $t_1$ be given by \eqref{t0t1} with
\begin{equation}\label{34}
i^*=\min\{i\in\{0, \ldots, n\}:|s^{y,v}_{\pi(i+1)}-s^{y,v}_{\pi(i)}|\geq\delta/(n+1)\}.
\end{equation}
Put
\[ C_1=C([0, t_0]\cup [t_1, \tau_1]),\quad C_2=C([t_0,
t_1]),\quad C_3=C([\tau_1, \tau_2]),
\] let $ \widetilde W\colon
C([0,T])\to C_1$ and $B\colon C([0,T])\to C_2$ be given by
\eqref{18} and \eqref{20}. respectively, and define  $\overline W\colon C([0,T])\to
C_3$ by
\[
\overline W(w)=w|_{[\tau_1,\tau_2]}
\]
for $w\in C([0,T])$. Moreover, define mappings $H_1\colon C_1
\to\R$, $H_2\colon C_2\to\R$, $H_3\colon C_3\to\R$ and $J\colon
C_1\times C_2\times C_3\to\R$ by
\begin{align*}
&H_1( \widetilde w)=\int_{0}^{t_0} f'(t)\, \widetilde w(t)dt+\int_{t_0}^{t_1} f'(t)\,\Bigl(\frac{t-t_0}{t_1-t_0} \widetilde w(t_1)+\frac{t_1-t}{t_1-t_0} \widetilde
w(t_0)\Bigr)dt+\int_{t_1}^{\tau_1} f'(t)\, \widetilde w(t)dt,\\
&H_2(b)=\int_{t_0}^{t_1} f'(t)\,b(t)dt, \quad H_3(\overline w)=-\int_{\tau_1}^{\tau_2} g'(t)\, \overline w(t)dt
\end{align*}
as well as
\[
J( \widetilde w,b, \overline w)=\gamma\cdot
\cos\bigl((H_1(\widetilde w)+H_2(b))\cdot \psi(H_3( \overline
w))\bigr)
\]
 for $\widetilde w\in C_1$, $b\in C_2$ and $\overline w\in C_3$. We
 then have
\[
F=J(\widetilde W, B, \overline W).
\]
Clearly, $ Q^{\overline W}_{y,v}$ is the dirac measure concentrated at $v|_{[\tau_1, \tau_2]}$.
 Using Lemma \ref{induced}(iii),(iv) and the triangle inequality we
thus obtain
\begin{align*}
&\int_{C([0,T])}  |
      F(w) -\text{pr}_4(G(y,v))
    |\, Q_{y,v}(dw)\\
    &\qquad\quad =\int_{C_1}\int_{C_2} |
      J(\widetilde w, b, v|_{[\tau_1, \tau_2]}) -\text{pr}_4(G(y,v))
    |\, Q^{B}_{y,v}(db)\,
    Q^{\widetilde W}_{y,v}(d\widetilde w)\\
    &\qquad\quad  =\int_{C_1}\int_{C_2} \frac{1}{2}\, \bigl(|
      J(\widetilde w, b, v|_{[\tau_1, \tau_2]})  -\text{pr}_4(G(y,v))
    |\\
    &\hspace{5cm}+|
      J(\widetilde w, -b, v|_{[\tau_1, \tau_2]})  -\text{pr}_4(G(y,v))
    |\bigr)\, Q^{B}_{y,v}(db)\,
    Q^{\widetilde W}_{y,v}(d\widetilde w)\\
    &\qquad\quad  \geq\int_{C_1}\int_{C_2}   \frac{1}{2}\,|
      J(\widetilde w, b, v|_{[\tau_1, \tau_2]})-
      J(\widetilde w, -b, v|_{[\tau_1, \tau_2]})
    |\, Q^{B}_{y,v}(db)\,
 Q^{\widetilde W}_{y,v}(d\widetilde w).
\end{align*}
Put $z=H_3(v|_{[\tau_1, \tau_2]})$. The fact that for all
$x,y\in\R$
\[
\cos(x)-\cos(y)=2\sin(\tfrac{y-x}{2}) \sin(\tfrac{y+x}{2})
\]
thus implies
\begin{align*}
&\int_{C([0,T])}  |
      F(w) -\text{pr}_4(G(y,v))
    |\, Q_{y,v}(dw)\\
    &\qquad\qquad \geq\gamma \int_{C_1}|
     \sin\bigl(H_1( \widetilde w)\cdot \psi(z)\bigr)|\,  Q^{\widetilde W}_{y,v}(d\widetilde w)\cdot \int_{C_2}
    |\sin\bigl(H_2(b)\cdot \psi(z)\bigr)
    |\, Q^{B}_{y,v}(db)\\
    &\qquad\qquad =\gamma \int_{\R}|
     \sin(y\cdot \psi(z))|\,  Q^{H_1(\widetilde W)}_{y,v}(dy)\cdot \int_{\R}
    |\sin(y\cdot \psi(z))
    |\, Q^{H_2(B)}_{y,v}(dy).    
\end{align*}
Due to Lemma \ref{induced}(i),(ii), $Q^{\widetilde W}_{y,v}$ and
$Q^{B}_{y,v}$ are Gaussian measures on $\mathfrak B(C_1)$ and
$\mathfrak B(C_2)$, respectively. Since the mappings $H_1$ and $H_2$
are linear and continuous we conclude that $Q^{H_1(\widetilde
W)}_{y,v}$ and $Q^{H_2(B)}_{y,v}$ are Gaussian measures on
$\mathfrak B(\R)$. Let $m_1, \sigma_1^2$ and  $m_2, \sigma_2^2$ denote the mean and the
variance of $Q^{H_1(\widetilde W)}_{y,v}$ and $Q^{H_2(B)}_{y,v}$,
respectively. Applying Lemma \ref{l2} we obtain
\[
\int_{\R}|
     \sin(y\cdot \psi(z))|\,  Q^{H_1(\widetilde W)}_{y,v}(dy)\geq \frac{ \exp \bigl(- \tfrac{\pi^2}{8}\bigr) }{ \sqrt{ 8 \pi } } 
  \cdot 1_{[1, \infty)}(\sigma_1\cdot\psi(z))
\]
as well as
\[
\int_{\R}
      |\sin(y\cdot \psi(z))
    |\, Q^{H_2(B)}_{y,v}(dy)\geq \frac{ \exp \bigl(- \tfrac{\pi^2}{8}\bigr) }{ \sqrt{ 8 \pi } } \cdot 1_{[1, \infty)}(\sigma_2\cdot\psi(z)).
\]
Hence
\begin{equation}\label{40}
\int_{C([0,T])}  |
      F(w) -\text{pr}_4(G(y,v))
    |\, Q_{y,v}(dw)\geq \frac{ \gamma  \exp\bigl(- \tfrac{\pi^2}{4}\bigr) }{  8 \pi  }
 \cdot 1_{[1, \infty)}\bigl(\min(\sigma_1,\sigma_2)\cdot\psi(z)\bigr).
\end{equation}
Next we derive lower bounds for $\sigma_1^2$ and $\sigma_2^2$.  Due to Lemma \ref{induced}(ii) we have
\[
m_2=\int_{C_2} H_2(b)\, Q^B_{y,v}(db)=\int_{t_0}^{t_1}f'(t)\cdot m(Q^{B}_{y,v})(t)\,dt=0
\]
as well as
\begin{align*}
\sigma_2^2&=\int_{C_2} (H_2(b)-m_2)^2\, Q^B_{y,v}(db)=\int_{t_0}^{t_1}\int_{t_0}^{t_1} f'(r)\cdot f'(t)\cdot R(Q^{B}_{y,v})(r,t)dr dt\\
&=\int_{t_0}^{t_1}\int_{t_0}^{t_1} f'(r)\cdot f'(t)\cdot \frac{(t_1-\max(r,t))\cdot (\min(r,t)-t_0)}{t_1-t_0}dr dt.
\end{align*}
It is easy to see that for all $a,b\in\R$ with $a<b$
\begin{equation}\label{30}
\int_{a}^{b}\int_{a}^{b} \frac{(b-\max(r,t))\cdot (\min(r,t)-a)}{b-a}dr dt=\frac{(b-a)^3}{12}.
\end{equation}
Moreover, the assumption $ \inf_{ t\in [ 0, \tau_1/2  ] } | f'(t) | > 0$ implies that for all $r,t\in[0, \tau_1/2]$
\begin{equation}\label{31}
f'(r)\cdot f'(t)=|f'(r)\cdot f'(t)|\geq \alpha.
\end{equation}
Observing \eqref{34} we thus obtain
\begin{equation}\label{35}
\sigma_2^2\geq \frac{\alpha(t_1-t_0)^3}{12}\geq  \frac{\alpha\delta^3}{12(n+1)^3}.
\end{equation}
Put $s_i=s^{y,v}_{\pi(i)}$ for $i=0, \ldots, n+1$ as well as  $s_{n+2}=\tau_1$. Clearly,
\begin{align*}
m_1&=\int_{C_1} H_1(\widetilde w)\, Q^{\widetilde W}_{y,v}(d\widetilde w )=\sum_{i\in \{0, \ldots, n+1\}\setminus \{i^*\}}\int_{s_i}^{s_{i+1}} f'(t) m(Q^{\widetilde W}_{y,v})(t)dt\\
&\hspace{5.3cm}+\int_{t_0}^{t_1} f'(t)\Bigl(\frac{t-t_0}{t_1-t_0}  m(Q^{\widetilde W}_{y,v})(t_1)+\frac{t_1-t}{t_1-t_0}  m(Q^{\widetilde W}_{y,v})(t_0)\Bigr)dt.
\end{align*}
Hence
\begin{align*}
\sigma_1^2&=\int_{C_1} (H_1(\widetilde w)-m_1)^2\, Q^{\widetilde W}_{y,v}(d\widetilde w )\\
&= \int_{C_1} \Bigl(\sum_{i\in \{0, \ldots, n+1\}\setminus \{i^*\}}\int_{s_i}^{s_{i+1}} f'(t)\cdot (\widetilde w(t)-m(Q^{\widetilde W}_{y,v})(t))dt\\
&\hspace{2cm} +\int_{t_0}^{t_1} f'(t)\cdot\frac{t-t_0}{t_1-t_0} \cdot (\widetilde w(t_1)-m(Q^{\widetilde W}_{y,v})(t_1))dt\\
&\hspace{4cm}+\int_{t_0}^{t_1}  f'(t)\cdot \frac{t_1-t}{t_1-t_0}\cdot (\widetilde w(t_0)-  m(Q^{\widetilde W}_{y,v})(t_0))dt\Bigr)^2 Q^{\widetilde W}_{y,v}(d\widetilde w ).
\end{align*}
Lemma \ref{induced}(i)  implies that $R(Q^{\widetilde W}_{y,v})(r, t)=0$ for all  $r\in[s_i, s_{i+1}]$ and $t\in [s_j, s_{j+1}]$ and all $i, j\in\{0, \ldots, n\}$ with $i\not =j$ 
as well as  $R(Q^{\widetilde W}_{y,v})(r, t)=0$ for all $r\in [0, \tau_1]$ and $t\in \{t_0, t_1\}\cup [s_{n+1}, s_{n+2}]$. Thus
\begin{align*}
\sigma_1^2
&=\sum_{i\in \{0, \ldots, n\}\setminus \{i^*\}}\int_{s_i}^{s_{i+1}}\int_{s_i}^{s_{i+1}} f'(r)\cdot f'(t)\cdot R(Q^{\widetilde W}_{y,v})(r,t)dr dt\\
&= \sum_{i\in \{0, \ldots, n\}\setminus \{i^*\}}\int_{s_i}^{s_{i+1}}\int_{s_i}^{s_{i+1}} f'(r)\cdot f'(t)\cdot \frac{(s_{i+1}-\max(r,t))\cdot (\min(r,t)-s_i)}{s_{i+1}-s_i}dr dt.
\end{align*}
Using \eqref{31}, \eqref{30} and the H\"older inequality we therefore obtain
\begin{align*}
\sigma_1^2
&\geq \sum_{i\in \{0, \ldots, n\}\setminus \{i^*\}}\int_{s_i}^{s_{i+1}}\int_{s_i}^{s_{i+1}} |f'(r)\cdot f'(t)|\cdot \frac{(s_{i+1}-\max(r,t))\cdot (\min(r,t)-s_i)}{s_{i+1}-s_i}dr dt\\
&\geq \frac{\alpha}{12}\sum_{i\in \{0, \ldots, n\}\setminus \{i^*\}} (s_{i+1}-s_i)^3\\
&\geq \frac{\alpha}{12 n^2}\cdot\bigl(\sum_{i\in \{0, \ldots, n\}\setminus \{i^*\}} (s_{i+1}-s_i) \bigr)^3= \frac{\alpha}{12 n^2}\cdot (\delta-(t_1-t_0))^3.
\end{align*}
The assumption that $\widehat X$ uses the evaluation site $\delta/2$ implies that $t_1-t_0\leq\delta/2$. Hence
\begin{equation}\label{360}
\sigma_1^2\geq  \frac{\alpha\delta^3}{96 n^2}.
\end{equation}
The desired lower bound \eqref{41} follows from \eqref{40}, \eqref{35} and \eqref{360}.

We conclude from \eqref{42},\eqref{13} \eqref{41} and \eqref{44} that
\begin{equation}\label{71}
\EE| X_4^{ \psi }( T ) -\widehat X_4|\geq \frac{ \gamma \exp \bigl(- \tfrac{\pi^2}{4}\bigr) }{  8 \pi  }\cdot
   \PP \Bigl(\Bigl\{\psi(X_3^\psi(\tau_2))\geq \sqrt{\tfrac{96 \nu^3}{\alpha \delta^3}}\Bigr\}\Bigr).
\end{equation}
Put
\[
A=\Bigl\{\psi(X_3^\psi(\tau_2))\geq \bigl(1+ \sqrt{\tfrac{96 }{\alpha \delta^3}}\bigr) N^3\Bigr\}, \quad B=\{\nu\leq N^2\}.
\]
Clearly,
\begin{align}\label{70}
\PP \Bigl(\Bigl\{\psi(X_3^\psi(\tau_2))\geq \sqrt{\tfrac{96 \nu^3}{\alpha \delta^3}}\Bigr\}\Bigr)&\geq \PP \Bigl(\Bigl\{\psi(X_3^\psi(\tau_2))\geq \sqrt{\tfrac{96 \nu^3}{\alpha \delta^3}}\Bigr\}\cap \{\nu\leq N^2\}\Bigr)\notag\\
&\geq \PP(A\cap B)\geq\PP(A)-\PP(B^c).
\end{align}
Note that $X_3^\psi(\tau_2)\sim  \mathcal{N}( 0, \beta ) $, where $\beta$ is given by \eqref{51}. Moreover, the assumption that $\psi$ is continuous with $\lim_{x\to\infty} \psi(x)=\infty$ and  $1\in\psi(\R)$ ensures that
\[
\bigl(1+ \sqrt{\tfrac{96 }{\alpha \delta^3}}\bigr) N^3\in\psi(\R).
\]
Hence
\begin{align}\label{72}
\PP (A)&=\PP\Bigl(\Bigl\{X_3^\psi(\tau_2)\geq \psi^{-1}\bigl(\bigl(1+ \sqrt{\tfrac{96 }{\alpha \delta^3}}\bigr)  N^3\bigr)\Bigr\}\Bigr)\notag\\
&\geq \frac{1}{\sqrt{2\pi\beta}}\int_{\psi^{-1}\bigl(\bigl(1+ \sqrt{\tfrac{96 }{\alpha \delta^3}}\bigr)  N^3\bigr)}^{1+\psi^{-1}\bigl(\bigl(1+ \sqrt{\tfrac{96 }{\alpha \delta^3}}  \bigr)N^3\bigr)} \exp\bigl(-\tfrac{x^2}{2\beta}\bigr) dx\notag\\
&\geq \frac{1}{\sqrt{2\pi\beta}}\cdot\exp\bigl(-\tfrac{1}{2\beta}\cdot \bigl(1+\psi^{-1}\bigl(\bigl(1+ \sqrt{\tfrac{96 }{\alpha \delta^3}}\bigr)  N^3\bigr)\bigr)^2\bigr)\notag\\
&\geq \frac{\exp\bigl(-\tfrac{1}{\beta}\bigr)}{\sqrt{2\pi\beta}}\cdot\exp\bigl(-\tfrac{1}{\beta}\cdot \bigl(\psi^{-1}\bigl(\bigl(1+ \sqrt{\tfrac{96 }{\alpha \delta^3}}\bigr)  N^3\bigr)\bigr)^2\bigr).
\end{align}
By the Markov inequality and the fact that $\widehat X\in\alg^\delta_{N+1}$,
\begin{equation}\label{73}
\PP(B^c)\leq \frac{\EE\nu}{N^2}\leq \frac{N+1}{N^2}\leq  \frac{2}{N}.
\end{equation}
Estimates \eqref{71}, \eqref{70}, \eqref{72} and \eqref{73} imply \eqref{61}, which completes the proof of the theorem in the case $\delta\in(0, \tau_1/2]$.

If $\delta\in (\tau_1/2,T]$ then the lower bound \eqref{61} follows from the  fact that $\alg_N^\delta\subseteq \alg_N^{\tau_1/2}$ and the lower bound \eqref{61} in the case $\delta=\tau_1/2$.

\section*{Acknowledgement}
I am
grateful to Thomas M\"uller-Gronbach for stimulating discussions
on the topic of this article.

\bibliographystyle{acm}
\bibliography{bibfile}

\end{document}